\documentclass{amsart}
\usepackage[T1]{fontenc}
\usepackage[utf8]{inputenc}
\usepackage{dsfont,amssymb}
\usepackage[mathscr]{eucal}
\usepackage{enumitem}
\usepackage{color}

\usepackage[bookmarks,colorlinks]{hyperref}

\hypersetup{
    linkcolor=brickred,
    citecolor=brickred
}

\definecolor{brickred}{cmyk}{0, 0.89, 0.94, 0.28}

\numberwithin{equation}{section}

\newtheorem{theorem}{Theorem}[section]
\newtheorem{lemma}[theorem]{Lemma}

\theoremstyle{definition}

\DeclareMathOperator{\id}{Id}
\DeclareMathOperator{\re}{Re}
\DeclareMathOperator{\im}{Im}
\DeclareMathOperator{\res}{Res}

\DeclareMathOperator{\ind}{\mathds{1}}

\newcommand{\dht}{\mathscr{H}}
\newcommand{\rtt}{\dht_{\mathrm{RT}}}
\newcommand{\kht}{\dht_{\mathrm{K}}}
\newcommand{\adp}{\dht_{\mathrm{ADP}}}
\newcommand{\intertwine}{\mathscr{K}}
\newcommand{\mdht}{\mathscr{J}}

\newcommand{\CA}{\mathscr{A}}
\newcommand{\CB}{\mathscr{B}}
\newcommand{\CC}{\mathscr{C}}
\newcommand{\CD}{\mathscr{D}}
\newcommand{\CE}{\mathscr{E}}
\newcommand{\CF}{\mathscr{F}}
\newcommand{\CG}{\mathscr{G}}
\newcommand{\CI}{\mathscr{I}}
\newcommand{\CK}{\mathscr{K}}

\newcommand{\CT}{\mathscr{T}}
\newcommand{\MA}{A}
\newcommand{\MH}{H}
\newcommand{\MI}{I}
\newcommand{\trans}{\star}

\newcommand{\pr}{\mathds{P}}
\newcommand{\ex}{\mathds{E}}
\newcommand{\C}{\mathds{C}}
\newcommand{\R}{\mathds{R}}
\newcommand{\Z}{\mathds{Z}}

\newcommand{\eps}{\varepsilon}

\newcommand{\abs}[1]{\left| #1 \right|}

\newcommand{\expr}[1]{\left( #1 \right)}

\renewcommand{\le}{\leqslant}
\renewcommand{\ge}{\geqslant}

\usepackage{ifthen}
\newcommand{\formula}[2][nolabel]%
{%
 \ifthenelse{\equal{#1}{nolabel}}%
 {\begin{align*} #2 \end{align*}}%
 {%
  \ifthenelse{\equal{#1}{}}%
  {\begin{align} #2 \end{align}}%
  {\begin{align} \label{#1} #2 \end{align}}%
 }%
}

\begin{document}

\title[Discrete Hilbert transform]{On the $\ell^p${-}norm of the Discrete Hilbert transform}
\author{Rodrigo Ba\~nuelos}
\address{Department of Mathematics \\ Purdue University \\ 150 N.\@ University Street, West Lafayette, IN 47907-2067, USA}
\email{banuelos@purdue.edu}
\author{Mateusz Kwa\'snicki}
\address{Faculty of Pure and Applied Mathematics \\ Wrocław University of Science and Technology \\ ul.\@ Wybrzeże Wyspiańskiego 27 \\ 50-370 Wrocław, Poland}
\email{mateusz.kwasnicki@pwr.edu.pl}
\thanks{R.~Ba\~nuelos is supported in part by NSF Grant \#1403417-DMS. M.~Kwa\'snicki is supported by the Polish National Science Centre (NCN) grant no.\@ 2015/19/B/ST1/01457.}
\date{\today}

\begin{abstract}
Using a representation of the discrete Hilbert transform in terms of martingales arising from Doob $h$-processes, 
we prove that its $\ell^p$-norm, $1<p<\infty$, is bounded above by the $L^p${-}norm of the continuous Hilbert transform. Together with the already known lower bound, this resolves the long{-}standing conjecture that the norms of these operators are equal. 
\end{abstract}

\maketitle

\tableofcontents

%
%

\section{Introduction and main results}

The \emph{discrete Hilbert transform} is an operator $\dht$ which maps the sequence $(a_n)$ to the sequence $(\dht a_n)$ defined by
\formula[eq:dht]{
 \dht a_n & = \frac{1}{\pi} \sum_{m \in \Z \setminus \{0\}} \frac{a_{n - m}}{m} \, .
}
This operator was introduced by D.~Hilbert in the first decade of 20th century, and it was proved by M.~Riesz in~\cite{bib:r27,bib:r28} and E.~C.~Titchmarsh in~\cite{bib:t26,bib:t28} that for $p \in (1, \infty)$ it is a bounded operator on $\ell^p$, the space of (doubly infinite) sequences $(a_n)$ with finite $\ell^p$-norm
\formula{
 \|(a_n)\|_p & = \expr{\sum_{n \in \Z} |a_n|^p}^{\!1/p}. 
}
See Section 8.12 in~\cite{bib:hlp34} for further details. The purpose of this paper is to evaluate its $\ell^p$-norm, henceforth denoted by $\|\dht\|_{p \rightarrow p}$. As shown by E.~Laeng (Theorem~4.3 in~\cite{bib:l07}), it is known that $\|\dht\|_{p \rightarrow p}$ is at least as large as the norm of the continuous Hilbert transform on $L^p(\R)$. The latter was found by S.~Pichorides in~\cite{bib:p72} to be equal to $\cot(\pi / (2 p^*))$, where $p^* = \max(p, p / (p - 1))$. The equality of these norms is a long-standing conjecture, initiated by an erroneous proof of E.~C.~Titchmarsh in~\cite{bib:t26}. In~\cite{bib:l07}, Theorem~4.5, its validity is proved for $p = 2^k$ or $p = 2^k / (2^k - 1)$ for $k = 1, 2, \ldots$\, The proof is attributed to I.~E.~Verbitsky. The following result proves the upper bound in full generality and together with the known lower bound settles the conjecture. 

\begin{theorem}\label{thm:dht}
Let $(a_n)$ be a sequence in $\ell^p$, $1<p<\infty$. Then 
\formula[eq:dht:norm]{
 \|(\dht a_n)\|_p & \le \cot\!\expr{\frac{\pi}{2 p^*}} \|(a_n)\|_p ,
}
where $p^* = \max(p, p / (p - 1))$. Consequently, $\|\dht\|_{p \rightarrow p} \le \cot(\pi / (2 p^*))$.
\end{theorem}

For a detailed discussion of the history of the above problem and related topics we refer the reader to E.~Laeng~\cite{bib:l07}, as well as to~\cite{bib:adp18,bib:ds16}. A similar question for second-order discrete Riesz transforms was resolved in~\cite{bib:dp14}, see also~\cite{bib:dop18}.

The literature on harmonic analysis in the discrete setting, and in particular on singular integrals and other classical operators, has a long history. In addition to the above papers, we refer to A.~Calder\'on and A.~Zygmund~\cite{CalZyg}, L.~B.~Pierce~\cite{Pie}, F.~Lust-Piquard~\cite{Piq} and E.~M.~Stein and S.~Wainger~\cite{SteWai1,SteWai2}, for a sample of some of this literature. 

It should be pointed out that there are other operators known under the name \emph{discrete Hilbert transform}. In particular, the version studied by M.~Riesz and E.~C.~Titchmarsh (and therefore sometimes called the Riesz--Titchmarsh transform),
\begin{equation}\label{Riesz1}
 \rtt a_n = \frac{1}{\pi} \sum_{m \in \Z} \frac{a_{n - m}}{m + 1/2}
\end{equation}
is often considered. Sometimes $m + 1/2$ in the denominator is replaced by $m - 1/2$. Another variant was introduced by S.~Kak in~\cite{bib:k70},
\begin{equation}\label{Riesz2}
 \kht a_n = \frac{2}{\pi} \sum_{m \in 2 \Z + 1} \frac{a_{n - m}}{m},
\end{equation}
where $2 \Z + 1$ denotes the set of odd integers. Unlike~\eqref{eq:dht}, $\rtt$ and $\kht$ are unitary operators on $\ell^2$. It is easy to see that both~\eqref{Riesz1} and~\eqref{Riesz2} lead to operators with equal norms on $\ell^p$. Finiteness of the $\ell^p$ norm of these operators was already proved by Riesz and Titchmarsh.

In \cite{CiaGilRonTor}, \'O.~Ciaurri, T.~A. Gillespie, L.~Roncal, J.~L. Torrea and J.~L.~Varona studied the Riesz--Titchmarsh operator $\rtt$ as a Riesz transforms of the discrete Laplacian on $\Z$,
\formula[torrea1]{
 \Delta_{\mathrm{d}} a_n & = a_{n+1}-2a_n+a_{n-1} .
}
Since $\Delta_{\mathrm{d}}=\tilde{\mathscr{D}}\mathscr{D}$, where $\mathscr{D}a_n=a_{n+1}-a_n$ and $\tilde{\mathscr{D}}a_n=a_n-a_{n-1}$, the corresponding Riesz transforms are naturally defined by 
\formula[torrea2]{
\mathscr{R}a_n & =\mathscr{D}(-\Delta_{\mathrm{d}})^{-1/2}a_n, &
\tilde{\mathscr{R}}a_n & =\tilde{\mathscr{D}}(-\Delta_{\mathrm{d}})^{-1/2}a_n.
}
Then $\mathscr{R} = \rtt$, and $\tilde{\mathscr{R}}$ is its variant with $m - 1/2$ in the denominator. Using the Poisson semigroup for the operator $\Delta_{\mathrm{d}}$, constructed in the usual way as a Bochner subordination of the heat semigroup, it is shown in~\cite{CiaGilRonTor} that $\mathscr{R}$ and $\tilde{\mathscr{R}}$ arise as boundary values of conjugate harmonic functions in the upper half-space $\Z\times\R_{+}$, which satisfy a version of Cauchy--Riemann equations. Furthermore, it is shown in~\cite{CiaGilRonTor} that if  $w$ is a (discrete) Muckenhoupt $A_p$--weight, $1 \le p < \infty$, then $\rtt = \mathscr{R}$ and $\tilde{\mathscr{R}}$ are bounded operators on $\ell^p(w)$, $1<p<\infty$, and map $\ell^1(w)$ to weak-$\ell^1(w)$. A more direct proof (without the Littlewood--Paley square functions in~\cite{CiaGilRonTor}) was given by R.~Hunt, B.~Muckenhoupt and R.~Wheeden in~\cite{HunMucWhe}, and for $w=1$ this follows from~\cite{CalZyg}. 

Finally, in~\cite{bib:adp18} N.~Arcozzi, K.~Domelevo and S.~Petermichl introduced the symmetrised Riesz--Titchmarsh transform, which is the average of the Riesz transforms $\mathscr{R}$, $\tilde{\mathscr{R}}$ studied in~\cite{CiaGilRonTor}. That is, they consider the operator
\formula[Riesz3]{
\begin{aligned}
 \adp a_n = \frac{1}{2} \, (\mathscr{R}+\tilde{\mathscr{R}}) a_n & = \frac{1}{2 \pi} \sum_{m \in \Z} \expr{\frac{a_{n - m}}{m + 1/2} + \frac{a_{n - m}}{m - 1/2}}\\
  & = \frac{1}{\pi} \sum_{m \in \Z} \frac{m a_{n - m}}{m^2 - 1/4}.
\end{aligned}
}
It has been conjectured that the norms of $\dht$ and $\rtt = \mathscr{R}$ (or $\tilde{\mathscr{R}}$, or $\kht$)
on $\ell^p$ are equal; see~\cite{bib:adp18,bib:ds16,bib:l07} for further discussion. Furthermore, there is also a version of $\rtt$ for any $\alpha \in (0, 1)$ replacing $1/2$, as discussed in~\cite{bib:l07}, with a natural conjecture on what its norm should be; see Conjecture~5.7 in~\cite{bib:l07}. Unfortunately, the proof of Theorem~\ref{thm:dht} cannot be easily adapted to resolve these conjectures.

The core part of the proof of Theorem~\ref{thm:dht} follows the probabilistic proof for the continuous Hilbert and Riesz transforms. In the classical setting of $\R^d$, the It\^o formula allows us to represent an $L^p(\R^d)$ function $f$ as a stochastic integral that involves a harmonic extension of $f$ to the upper half-space $\R^{d+1}_{+}=\R^d \times (0, \infty)$ and a $(d + 1)$-dimensional Brownian motion. In dimension one, the representation of the continuous Hilbert transform of $f$ as a martingale transform of this stochastic integral is well-known due to the Cauchy--Riemann equations and the fact that harmonic functions composed with Brownian motion produce martingales. In higher dimensions, R.~F.~Gundy and N.~T.~Varopoulos in~\cite{bib:gv79} expressed the Riesz transform of $f$ in a similar way, this time as the conditional expectation of a martingale transform of the stochastic integral representing $f$. This case naturally leads to the notion of differentially subordinate martingales with an additional probabilistic orthogonality property. In~\cite{bib:bw95}, R.~Ba\~nuelos and G.~Wang extended Burkholder's celebrated inequalities for martingale transforms under the orthogonality assumption and gave a probabilistic proof of the result of Pichorides for the Hilbert transform ($d = 1$) and its extension to Riesz transform ($d > 1$) due to T.~Iwaniec and G.~Martin~\cite{IwMa}. We also quote here closely related articles of I.~E.~Verbitsky~\cite{bib:v80:verbitsky} and M.~Ess\'en~\cite{bib:e84}. For more on orthogonal martingales and applications, we refer the reader to Chapter~6 of the monograph of A.~Os\k{e}kowski~\cite{bib:o12} and Sections~2.2, 3.1 and~3.2 of the survey article by R.~Ba\~nuelos~\cite{bib:b10}.

In order to prove Theorem~\ref{thm:dht} we use a similar method, but we build the pair of martingales using two-dimensional Brownian motion conditioned to hit the boundary of the upper half-plane $\R^2_{+}$ at a lattice point, instead of the usual Brownian motion used in the classical case. This requires certain modifications in the argument: we use Doob $h$-transforms of the Brownian motion similar to those used by R.~Ba\~nuelos in~\cite{bib:b86} in the continuous case. This construction leads to an operator $\mdht$ that is different from $\dht$, as well as from $\rtt = \mathscr{R}$, $\tilde{\mathscr{R}}$, $\kht$ and $\adp$. The corresponding estimate is stated in the following result.

\begin{theorem}\label{thm:mdht}
Let
\formula{
 \mdht_n & = \frac{1}{\pi n} \expr{1 + \int_0^\infty \frac{2 y^3}{(y^2 + \pi^2 n^2) \sinh^2 y} \, dy}
}
for $n \ne 0$, and $\mdht_0 = 0$. Let $(a_n)$ be a sequence in $\ell^p$, $1<p< \infty$, and let
\formula{
 \mdht a_n & = \sum_{m \in \Z} \mdht_m a_{n - m} .
}
Then
\formula[eq:mdht:norm]{
 \|(\mdht a_n)\|_p & \le \cot\!\expr{\frac{\pi}{2 p^*}} \|(a_n)\|_p ,
}
where $p^* = \max(p, p / (p - 1))$. The constant in the above inequality is best possible, and consequently $\|\mdht\|_{p \rightarrow p} = \cot(\pi / (2 p^*))$.
\end{theorem}

As we shall see below, the kernel $(\mdht_n)$ arises naturally when taking the conditional expectation of the martingale transform constructed via the same matrix that gives the continuous Hilbert transform. This kernel clearly dominates the kernel of the discrete Hilbert transform $\dht$, in the sense that $\mdht_n > 1 / (\pi n) > 0$ when $n > 0$ and $\mdht_n < 1 / (\pi n) < 0$ when $n < 0$. This alone, however, does not mean that the norm of $\dht$ on $\ell^p$ is bounded by the corresponding norm of $\mdht$. In order to derive Theorem~\ref{thm:dht} from Theorem~\ref{thm:mdht}, we prove that the discrete Hilbert transform is the composition of $\mdht$ and a convolution operator with probability kernel. This is a completely new phenomenon, which, to the best of our knowledge, has not appeared in the copious applications of martingale transform inequalities to singular integrals, especially to Riesz transforms in wide geometric settings, including $\R^d$, Wiener space (the Ornstein--Uhlenbeck operator), Lie groups and manifolds; see, for example, \cite{BanBau}, \cite{BanOse},  and references therein. In all previous such applications where optimal (or near optimal) constants are obtained, the operators are exactly given as projections (conditional expectations) of martingale transforms. That is, the analytic operators are factored as the composition of two operators: the martingale transform, which gives the optimal $L^p$ bound, and conditional expectation, which preserves the bound. In the present case of the discrete Hilbert transform these two operators are followed by a third one, a convolution with a probability kernel, which also does not increase the norm. The following lemma makes this precise.

\begin{lemma}\label{lem:dht:mdht}
If $(a_n)$ is a sequence in $\ell^p$ for some $p \in (1, \infty)$, then
\begin{equation}\label{proker1}
 \dht a_n = \sum_{m \in \Z} \intertwine_{n - m} \mdht a_m
\end{equation}
for an appropriate nonnegative sequence $(\intertwine_n)$ with total mass $1$.
\end{lemma}

By convexity of the $\ell^p$--norm, \eqref{proker1} implies that $\|(\dht a_n)\|_p \le \|(\mdht a_n)\|_p$, so that $\|\dht\|_{p \rightarrow p} \le \|\mdht\|_{p \rightarrow p}$. In other words, the estimate~\eqref{eq:mdht:norm} implies the estimate~\eqref{eq:dht:norm}. 

The discrete Laplacian $\Delta_{\mathrm{d}}$ is naturally associated with a simple symmetric random walk $S_n$ on $\Z$. The corresponding continuous-time random walk $X(t)$ is given by $X(t) = S_{N(t)}$, where $N(t)$ is a rate~$1$ Poisson process, independent of $S_n$. Following the construction of Gundy--Varopoulos, with the continuous-time simple symmetric random walk $X_t$ used in place of the standard Brownian motion, Arcozzi, Domelevo and Petermichl in~\cite{bib:adp18} give a representation of $\adp$ in terms of \emph{semi-discrete} martingales. This, however, is insufficient for an application of the Ba\~nuelos{--}Wang inequality due to the lack of the right orthogonality, which requires that the paths of the martingales be continuous; see Remark~6.1 on p.~246 in~\cite{bib:o12}. Finally, we note that for the same reason, the classical relation of the discrete Hilbert transform with discrete harmonic (or discrete analytic) functions is insufficient to obtain the optimal constant; for the definitions and more information about discrete harmonic and analytic functions, we refer to~\cite{bib:d56}.

This paper is organized as follows. Theorems~\ref{thm:dht} and~\ref{thm:mdht} are proved in Section~\ref{sec:proof}. Specifically, Sections~\ref{sec:burk}--\ref{sec:mdht} contain the proof of Theorem~\ref{thm:mdht}, while Lemma~\ref{lem:dht:mdht} and Theorem~\ref{thm:dht} are derived in Section~\ref{sec:dht}. Section \ref{sec:weak-type} raises questions concerning the best constant in the weak-type $(1, 1)$ inequality for the discrete Hilbert transform. To make the presentation clearer, we move most technical details to Section~\ref{sec:tech}. 

%
%

\section{The Hilbert transform as a projection of a martingale transform}
\label{sec:proof}

This section contains the core part of the proof of Theorem~\ref{thm:dht}. Our goal is to express the sequence $(\mdht a_n)$ as the conditional expectation (projection) of a martingale transform of a martingale constructed from the sequence $(a_n)$. These two martingales will have the additional property of orthogonality. With this, and the now classical $L^p$ inequalities for orthogonal martingales, Theorem~\ref{thm:mdht} will follow. We then use it for the proof of Theorem~\ref{thm:dht} as described in the introduction. The construction of these martingales, via Doob $h$-processes, is similar to the results of R.~Ba\~nuelos~\cite{bib:b86} in the continuous case. However, the identification of the kernel in the present case is much more intricate. To make the argument easier to follow, we postpone the more technical calculations to the next section.

We use standard notation for sequences: $(a_n)$ denotes the entire sequence, while $a_n$ is its $n$-th term. All sequences are doubly infinite: they are indexed by integers. We also use the same symbol to denote linear operators $\CA${,} their matrices $(\CA_{nm})$ and, if they are convolution operators{,} their convolution kernels $(\CA_n)$.

\subsection{Martingale transforms and their $L^p$ inequalities}
\label{sec:burk}

Let $p_n$ denote the Poisson kernel for the upper half-plane $\R^2_{+} = \R \times (0, \infty)$ with pole at $2 \pi n$, $n \in \Z$. That is,
\formula{
 p_n(x, y) & = \frac{1}{\pi} \, \frac{y}{(x - 2 \pi n)^2 + y^2}, 
}
for $x \in \R$ and $y > 0$. We define
\formula[eq:h]{
 h(x, y) & = \sum_{n \in \Z} p_n(x, y) = \frac{1}{2 \pi} \, \frac{\sinh y}{\cosh y - \cos x}. 
}
The fact that this sum is given by the right-hand side is proved in Lemma~\ref{lem:csc} below. Then $p_n$ and $h$ are positive harmonic functions in $\R^2_{+}$. For a given sequence $(a_n)$ with finitely many non-zero elements, we let
\formula{
 u(x, y) & = \sum_{n \in \Z} a_n \frac{p_n(x, y)}{h(x, y)} \, .
}
Then $u$ is $h$-harmonic in $\R^2_{+}$. That is, $\Delta(h u) = 0$, or, equivalently,
\formula{
 \Delta u(x, y) + \frac{\nabla u(x, y) \cdot \nabla h(x, y)}{h(x, y)} & = 0
}
when $x \in \R$, $y > 0$. Furthermore, using the explicit expressions for $p_n$ and $h$, we easily see that $u$ extends continuously to the boundary, and $u(2 \pi n, 0) = a_n$, for $n \in \Z$. That is, $u$ is the $h$-harmonic extension of the sequence $(a_n)$ to the upper half-space. 

We refer the reader to Chapter~III in~\cite{bib:b95} and the classical article~\cite{bib:cm83} for the basic construction, properties and stochastic calculus, of Doob $h$-processes.
 
Our next goal is to show, as in the classical case, that if we compose this harmonic function with an $h$-Brownian motion, we get a martingale. Let $Z_t$ be the Doob $h$-conditioned two-dimensional Brownian motion killed upon leaving the upper half-plane. That is, $Z_t$ is the Brownian motion conditioned to hit the boundary at one of the points of $2 \pi \Z \times \{0\}$. In fact, $Z_t$ hits $2 \pi \Z \times \{0\}$ in a finite time $\zeta$ with probability one, and $\zeta$ is the lifetime of $Z_t$. Furthermore,
\formula{
 dZ_t & = dB_t + \frac{\nabla h(Z_t)}{h(Z_t)} \, dt
}
for $t \in [0, \zeta)$, where $B_t$ is some two-dimensional Brownian motion. Clearly, the quadratic variation of $Z_t$ satisfies
\formula{
 d[Z]_t & = d[B]_t = \id dt ,
}
where $\id$ is the $2 \times 2$ identity matrix. For $t \in [0, \zeta)$ we let
\formula{
 M_t & = u(Z_t) .
}
By the It\^o's formula,
\formula{
 M_t & = M_0 + \int_0^t \nabla u(Z_s) \cdot dZ_s + \int_0^t \Delta u(Z_s) ds \\
 & = M_0 + \int_0^t \nabla u(Z_s) \cdot dB_s + \int_0^t \expr{\Delta u(Z_s) + \frac{\nabla u(Z_s) \cdot \nabla h(Z_s)}{h(Z_s)}} ds .
}
Since $u$ is $h$-harmonic, we simply have
\formula{
 M_t & = M_0 + \int_0^t \nabla u(Z_s) \cdot dB_s ,
}
and thus $M_t$, $t \in [0, \zeta)$, is a martingale.

Following the standard notation, for any $2 \times 2$ matrix $\MA$ with constant coefficients we define the martingale transform of $M_t$ by
\formula[mtran1]{
\begin{aligned}
 \MA \trans M_t & = \int_0^t \MA \nabla u(Z_s) \cdot dB_s \\
 & = \int_0^t \MA \nabla u(Z_s) \cdot dZ_s - \int_0^t \frac{\MA \nabla u(Z_s) \cdot \nabla h(Z_s)}{h(Z_s)} \, ds, 
\end{aligned}
}
for $t \in [0, \zeta)$. Comparing the quadratic variation of the martingales $M_t$ and $\MA \trans M_t$ we see that 
\formula{
 [\MA \trans M]_t & = \int_0^t |\MA \nabla u(Z_s)|^2 ds \le \|\MA\|^2 \int_0^t |\nabla u(Z_s)|^2 ds \le \|\MA\|^2 [M]_t ,
}
where $\|\MA\|$ is the operator norm of the matrix $\MA$. Thus, the martingale transform $\MA \trans M_t$ is \emph{differentially subordinate} to the martingale $\|\MA\| M_t$; see~\cite{bib:b10,bib:o12}. If in addition $\MA \vec{v} \cdot \vec{v} = 0$ for all vectors ${\vec{v}} \in \R^2$, we also have that
\formula{
 [M, \MA \trans M]_t & = \int_0^t\nabla u(Z_s) \cdot \MA \nabla u(Z_s) ds = 0
}
and the martingales $\|\MA\| M_t$ and $\MA \trans M_t$ are said to be \emph{orthogonal}. 

With these definitions we have the following inequalities of Burkholder~\cite{bib:b84} and Ba\~nuelos--Wang~\cite{bib:bw95}, respectively:
\begin{enumerate}[label=(\roman*)]
\item For any $2 \times 2$ matrix $\MA$ and $p \in (1, \infty)$, 
\formula[eq:sub]{
 \|\MA \trans M_{\zeta-}\|_p \le (p^*-1) \|\MA\| \|M_{\zeta-}\|_p ,
}
where $p^* = \max(p, p / (p - 1))$.
\item For any $2 \times 2$ matrix $\MA$ satisfying $\MA \vec{v} \cdot \vec{v} = 0$ for all vectors ${\vec{v}} \in \R^2$ and for any $p \in (1, \infty)$,
\formula[eq:ortho]{
 \|\MA \trans M_{\zeta-}\|_p & \le \cot\!\expr{\frac{\pi}{2 p^*}} \|\MA\| \|M_{\zeta-}\|_p.
}
\end{enumerate}

For general martingales these inequalities are sharp, although this fact will not be used in this paper. 

\subsection{Conditioning upon the final location}
\label{sec:mdht}

We now construct a family of \emph{discrete singular integral} operator{s} using the martingale transforms $\MA \trans M_t$, where $\MA$ is a $2\times 2$ matrix as above. This is very similar to the results of Ba\~nuelos~\cite{bib:b86} in the continuous case. At its lifetime, $Z_t$ approaches $2 \pi n$ with probability $$h_n(x_0, y_0) = p_n(x_0, y_0) / h(x_0, y_0),$$ where $(x_0, y_0)$ is the starting point of $Z_t$. Conditioning upon this event leads to the usual Brownian motion conditioned to hit $(2 \pi n, 0)$ upon leaving the upper half-plane. Indeed, the Doob $h_n$-conditioned process $Z_t$ (which itself is the Doob $h$-conditioned Brownian motion) is precisely the Doob $(h_n h)$-conditioned Brownian motion, and $h_n h = p_n$. We define
\formula[cond1]{
 \CT_{\!\MA}a_n & = \ex_{(x_0, y_0)}\bigl(\MA \trans M_{\zeta-} \, \big\vert \, Z_{\zeta-} = (2 \pi n, 0)\bigr) .
}
Note that if $\MA=\MI$, the identity matrix, then $T_{\MI}$ is just the identity operator. Furthermore, since the conditional expectation is a contraction on $L^p$, $p \in (1, \infty)$, it follows from \eqref{eq:sub} and \eqref{eq:ortho}, respectively, that 
\formula{
 \|({\CT}_{\!\MA}a_n)\|_p & \le (p^*-1) \|\MA\|\|(a_n)\|_p,
}
for any $\MA$, and 
\formula[ortho-A]{
 \|({\CT}_{\!\MA}a_n)\|_p & \le \cot\!\expr{\frac{\pi}{2 p^*}} \|\MA\|\|(a_n)\|_p, 
}
if $\MA \vec{v} \cdot \vec{v} = 0$ for all vectors $\vec{v} \in \R^2$.
In particular, if 
\formula{
 \MH & = \left[\begin{matrix} 0 & -1 \\ 1 & 0 \end{matrix}\right],
}
we have 
\begin{equation}\label{Hilbert-H}
\|({\CT}_{\!\MH}a_n)\|_p \le \cot\!\expr{\frac{\pi}{2 p^*}} \|(a_n)\|_p .
\end{equation}
For simplicity, from now on we write $\CT$ instead of $\CT_{\!\MH}$.

In the classical case, this construction leads to the probabilistic representation of the Hilbert transform on $\R$ and the sharp bound is a direct consequence of \eqref{ortho-A}; see Ba\~nuelos~\cite{bib:b10}. In our present situation we note that the definition of the operator ${\CT}$ depends on the starting point of $Z_t$. We will prove that, in the limiting case as the starting point converges to infinity, ${\CT}$ coincides with the operator $\mdht$ introduced in Theorem~\ref{thm:mdht}. We remark that the limiting case corresponds to the \emph{background radiation process}, which for the usual (unconditioned) Brownian motion is discussed in~\cite{bib:bw95,bib:gv79,bib:v80:varopoulos}. To avoid unnecessary technical complications, however, we work with a finite starting point, and pass to the limit at the very end. The inequality~\eqref{Hilbert-H} will be preserved under this limit by the Fatou's lemma, because all matrix elements of $\CT$ will converge to the corresponding elements of $\mdht$.

As it was observed above, conditionally on the event $Z_{\zeta-} = (2 \pi n, 0)$, $Z_t$ is the Brownian motion in $\R^2_{+}$ conditioned to exit at $(2 \pi n, 0)${,} which we identify with $2\pi n$. Denote this process by $Z_t^n$ and the corresponding expectation by $\ex_{(x_0, y_0)}^{{n}}$. Using the right-hand side of \eqref{mtran1} with $\MA=\MH$, we have that~\eqref{cond1} is the same as 
\formula{
 \CT a_n & = \ex_{(x_0, y_0)}\expr{\int_0^\zeta \MH\nabla u(Z_t) \cdot dZ_t - \int_0^\zeta \frac{\MH\nabla u(Z_t) \cdot \nabla h(Z_t)}{h(Z_t)} \, dt \, \bigg| \, Z_{\zeta-} = (2 \pi n, 0)} \\
 & = \ex_{(x_0, y_0)}^{{n}}\expr{\int_0^\zeta \MH\nabla u(Z_t^n) \cdot dZ_t^n - \int_0^\zeta \frac{\MH\nabla u(Z_t^n) \cdot \nabla h(Z_t^n)}{h(Z_t^n)} \, dt} .
}
However, $Z_t^n$ can be represented as some Brownian motion, that we denote by $B_t^n$, with drift $\nabla p_n(Z_t^n)$. This leads to
\formula{
 \CT a_n & = \ex_{(x_0, y_0)}^{{n}}\expr{\int_0^\zeta \MH\nabla u(Z_t^n) \cdot d{B_t^n} + \int_0^\zeta \MH\nabla u(Z_t^n) \cdot \expr{\frac{\nabla p_n(Z_t^n)}{p_n(Z_t^n)} - \frac{\nabla h(Z_t^n)}{h(Z_t^n)}} dt} \\
 & = \ex_{(x_0, y_0)}^{{n}}\expr{\int_0^\zeta \MH\nabla u(Z_t^n) \cdot \expr{\frac{\nabla p_n(Z_t^n)}{p_n(Z_t^n)} - \frac{\nabla h(Z_t^n)}{h(Z_t^n)}} dt} .
}

Our task now is to write an analytic expression for this expectation that leads to our kernel $\mdht_n$. This is accomplished by employing the occupation time formula for the processes $Z_t^n$. First, we denote by $G(x, y) dy$ the occupation measure of the Brownian motion killed upon leaving the upper half-plane; in other words, $G(x, y)$ is the Green function for the upper half-plane for $-\tfrac{1}{2} \Delta$, with pole at the starting point $(x_0, y_0)$:
\formula{
 G(x, y) & = \frac{1}{2 \pi} \log \frac{(x - x_0)^2 + (y + y_0)^2}{(x - x_0)^2 + (y - y_0)^2} \, .
}
Then the occupation measure of the Doob $p_n$-conditioned Brownian motion $Z_t^n$ is $(p_n(x, y) / p_n(x_0, y_0)) G(x, y) dy$. It follows that
\formula{
 \CT a_n & = \int_{-\infty}^\infty \int_0^\infty \frac{p_n(x, y) G(x, y)}{p_n(x_0, y_0)} \, \MH\nabla u(x, y) \cdot \expr{\frac{\nabla p_n(x, y)}{p_n(x, y)} - \frac{\nabla h(x, y)}{h(x, y)}} dy dx .
}
Recall that $({\CT}a_n)$ is the image of $(a_n)$ under the linear operator $\CT$. Once again, we stress that this operator depends on the starting point $(x_0, y_0)$, and it will coincide with the operator $\mdht$ in Theorem~\ref{thm:mdht} only in the limiting case. In order to evaluate the matrix entries $\CT_{nm}$ of $\CT$ (and eventually $\mdht_{nm}$ of $\mdht$) for a fixed $m \in \Z$, we set $a_n = 1$ if $n = m$ and $a_n = 0$ otherwise. In this case $u(x, y) = p_m(x, y) / h(x, y)$, and ${\CT}_{nm} = \CT a_n$. Since $\nabla u = h^{-1} \nabla p_m - h^{-2} p_m \nabla h$, we obtain
\formula{
 \CT_{nm} & = \int_{-\infty}^\infty \int_0^\infty \frac{p_n(x, y) G(x, y)}{p_n(x_0, y_0)} \times \\
 & \hspace*{1em} \times \expr{\frac{\MH\nabla p_m(x, y)}{h(x, y)} - \frac{p_m(x, y) \MH\nabla h(x, y)}{(h(x, y))^2}} \cdot \expr{\frac{\nabla p_n(x, y)}{p_n(x, y)} - \frac{\nabla h(x, y)}{h(x, y)}} dy dx .
}
Elementary simplification (that involves orthogonality of $\nabla h$ and $\MH\nabla h$ and identities $\nabla h^{-1} = -h^{-2} \nabla h$ and $\MH\nabla h^{-1} \cdot \nabla p_n = -\MH\nabla p_n \cdot \nabla h^{-1}$) leads to
\formula[eq:mdht:nm]{\begin{aligned}
 \CT_{nm} = \int_{-\infty}^\infty \int_0^\infty \frac{G(x, y)}{p_n(x_0, y_0)} \, \bigl( & h^{-1}(x, y) \MH\nabla p_m(x, y)) \cdot \nabla p_n(x, y) \\
 \phantom{\int_{-\infty}^\infty} & {} + {} p_n(x, y) \MH\nabla p_m(x, y) \cdot \nabla h^{-1}(x, y) \\
 \phantom{\int_{-\infty}^\infty} & {} - {} p_m(x, y) \MH\nabla p_n(x, y) \cdot \nabla h^{-1}(x, y) \bigr) dy dx .
\end{aligned}}
Before further simplification, we set $x_0 = 0$ and consider the limit as $y_0 \to \infty$. In this case $G(x, y) / p_n(x_0, y_0)$ converges to $2 y$, and the integral that defines ${\CT}_{nm}$ converges to a finite limit{, that we denote $\mdht_{nm}$}; see Lemma~\ref{lem:uni} for details. Therefore, in the limiting case we have
\formula{
 \mdht_{nm} = \int_{-\infty}^\infty \int_0^\infty 2 y \bigl( & h^{-1}(x, y) \MH\nabla p_m(x, y) \cdot \nabla p_n(x, y) \\
 \phantom{\int_{-\infty}^\infty} & {} + {} p_n(x, y) \MH\nabla p_m(x, y) \cdot \nabla h^{-1}(x, y) \\
 \phantom{\int_{-\infty}^\infty} & {} - {} p_m(x, y) \MH\nabla p_n(x, y) \cdot \nabla h^{-1}(x, y) \bigr) dy dx .
}
If $n = m$, the second and third terms cancel and the first is zero due to the orthogonality of $\nabla p_n$ and $\MH \nabla p_n$. Hence, $\mdht_{n n} = 0$. Otherwise, we can split the above integral into three:
\formula{
 \mdht_{nm} & = \int_{-\infty}^\infty \int_0^\infty 2 y h^{-1}(x, y) \MH\nabla p_m(x, y) \cdot \nabla p_n(x, y) dy dx \\
 & + \int_{-\infty}^\infty \int_0^\infty 2 y p_n(x, y) \MH\nabla p_m(x, y) \cdot \nabla h^{-1}(x, y) dy dx \\
 & - \int_{-\infty}^\infty \int_0^\infty 2 y p_m(x, y) \MH\nabla p_n(x, y) \cdot \nabla h^{-1}(x, y) dy dx ;
}
finiteness of each of the integrals in the right-hand side follows from the proof of Lemma~\ref{lem:uni}. A substitution $x = 2 \pi (n + m) - x'$ reduces the last integral to the negative of the middle one: we have $h(x, y) = h(x', y)$, $p_m(x, y) = p_n(x', y)$, and $\MH\nabla_{x,y} p_n(x, y) \cdot \nabla_{x,y} h^{-1}(x, y) = -\MH\nabla_{x',y} p_m(x', y) \cdot \nabla_{x',y} h^{-1}(x', y)$. Therefore,
\formula[eq:mdht:n]{
\begin{aligned}
 \mdht_{nm} & = \int_{-\infty}^\infty \int_0^\infty 2 y h^{-1}(x, y) \MH\nabla p_m(x, y) \cdot \nabla p_n(x, y) dy dx \\
 & + \int_{-\infty}^\infty \int_0^\infty 4 y p_n(x, y) \MH\nabla p_m(x, y) \cdot \nabla h^{-1}(x, y) dy dx .
\end{aligned}
}
By considering a substitution $x = 2 \pi n + x'$, we see that $\mdht_{nm}$ depends only on $n - m$. For this reason we write $\mdht_{nm} = \mdht_{n - m}$. With this notation, $\mdht$ is a convolution operator with kernel $(\mdht_n)$.

The expression for $\mdht_n = \mdht_{n0}$ obtained above apparently cannot be further simplified using soft methods. After substituting the explicit expressions for $p_n$, $p_0$ and $h$, we will be able to evaluate the integral in $x$ explicitly. This will be done in Lemma~\ref{lem:int}. The final result is given by
\formula{
 \mdht_n & = \frac{1}{\pi n} \expr{1 + \int_0^\infty \frac{2 y^3}{(y^2 + \pi^2 n^2) \sinh^2 y} \, dy}
}
for $n \ne 0$, and $\mdht_0 = 0$.

As remarked above, the inequality~\eqref{Hilbert-H} holds in the limiting case with $x_0 = 0$ and $y_0 \to \infty$. We conclude that if $p \in (1, \infty)$ and $({\mdht}a_n)$ is the convolution of $(a_n)$ and $(\mdht_n)$, then
\formula{
 \|(\mdht a_n)\|_p \le \cot\!\expr{\frac{\pi}{2 p^*}} \|(a_n)\|_p .
}
This proves the inequality~\eqref{eq:mdht:norm} of Theorem~\ref{thm:mdht} for sequences $(a_n)$ with finitely many non-zero elements. Extension to general sequences $(a_n)$ in $\ell^p$ is quite simple: such a sequence can be approximated in $\ell^p$ by sequences $(a_n^{(k)})$ with finitely many non-zero elements, and since $(\mdht_n)$ belongs to $\ell^q$ for the dual exponent $q = p / (p - 1)$, the corresponding images $(\mdht a_n^{(k)})$ converge pointwise. Extension follows now from Fatou's lemma and this completes the proof of Theorem~\ref{thm:mdht}.

\subsection{Reduction to the discrete Hilbert transform, proof of Lemma \ref{lem:dht:mdht}}
\label{sec:dht}

Let $\dht$ denote the discrete Hilbert transform, the convolution operator with kernel $\dht_n = (\pi n)^{-1}$. Throughout this section we identify convolution operators $\CA$ with their kernels $(\CA_n)$, and so if $\CA$ and $\CB$ are two convolution operators, then $\CA \CB$ is also a convolution operator, whose kernel is the convolution of $(\CA_n)$ and $(\CB_n)$. We also denote by $\CI$ the identity operator: the convolution with kernel $\CI_0 = 1$, $\CI_n = 0$ for $n \ne 0$.

By Lemma~\ref{lem:ihj}, there is an absolutely summable sequence $(\CE_n)$ such that
\formula[eq:ihj]{
 \mdht & = \dht + \dht \CE .
}
Furthermore, $\CE_n < 0$ for all $n \ne 0$, $\CE_0 > 0$ and the sum of all $\CE_n$ is zero. Write $\alpha = (1 + \CE_0)^{-1}$, $\CG_n = -\alpha \CE_n$ for $n \ne 0$ and $\CG_0 = 0$. Then $\CG_n \ge 0$ for all $n$, the sum of all $\CG_n$ is equal to $\alpha \CE_0 = 1 - \alpha$, and
\formula{
 \mdht & = \dht + (\CE_0 \dht - \alpha^{-1} \dht \CG) = \alpha^{-1} \dht (\CI - \CG) .
}
The norm of $\CG$ as an operator on $\ell^p$ does not exceed $\|(\CG_n)\|_1 = 1 - \alpha$. Similarly, the norm of $\CG^k$, the $k$-th power of $\CG$, does not exceed the sum of all elements of the $k$-fold convolution of $(\CG_n)$, which is equal to $(1 - \alpha)^k$. It follows that we can define the operator
\formula{
 \CK & = \alpha \sum_{k = 0}^\infty \CG^k .
}
Furthermore, $\CK$ is a convolution operator with kernel $(\CK_n)$ such that $\CK_n \ge 0$ for all $n$, and the sum of all $\CK_n$ is equal to $\alpha \sum_{k = 0}^\infty (1 - \alpha)^k = 1$. Finally,
\formula{
 \mdht \CK & = \alpha^{-1} \dht (\CI - \CG) \sum_{k = 0}^\infty \alpha \CG^k = \dht \sum_{k = 0}^\infty (\CI - \CG) \CG^k = \dht .
}
This proves Lemma~\ref{lem:dht:mdht}.

By Jensen's inequality, for any $p \in (1, \infty)$ and any sequence $(b_n)$ in $\ell^p$,
\formula{
 \|(\CK_n) * (b_n)\|_p & \le \|(b_n)\|_p .
}
Therefore,
\formula{
 \|(\dht a_n)\|_p & = \|(\CK_n) * (\mdht a_n)\|_p \le \|(\mdht a_n)\|_p \le \cot\!\expr{\frac{\pi}{2 p^*}} \|(a_n)\|_p .
}
This proves the estimate~\eqref{eq:dht:norm} of Theorem~\ref{thm:dht}. 

\subsection{On the weak-type inequality}\label{sec:weak-type}

In \cite{Dav}, B.~Davis found the best constant in Kolmogorov's weak-type $(1, 1)$ inequality for the classical Hilbert transform $H$ on $\R$, $H f(x) = \pi^{-1} \int_{-\infty}^\infty y^{-1} f(x - y) dy$, with the Cauchy principal value integral. More precisely, he proved that for $f\in L^1(\R)$, 
\begin{equation}\label{davis}
m\{x\in \R: |Hf(x)|>\lambda\} \le \frac{D}{\lambda}\int_{\R}|f(x)| dx
\end{equation} 
where $m$ is the Lebesgue measure and
\begin{equation}\label{catalan}
D=\frac{1+\frac{1}{3^2}+\frac{1}{5^2}+\frac{1}{7^2}+\frac{1}{9^2}+
\cdots}{1-\frac{1}{3^2}+\frac{1}{5^2}-\frac{1}{7^2}+\frac{1}{9^2}-\cdots}=\frac{\pi^2}{8\beta(2)} \, ,
\end{equation}
with $\beta(2)$ the so called \emph{Catalan's constant}, and that the inequality~\eqref{davis} is sharp. 

While the sharp $\ell^p$, $1<p<\infty$, version of Pichorides's inequality for the discrete Hilbert transform has been investigated by many authors, as already discussed above, the problem of proving weak-type $(1, 1)$ version of Davis's inequality does not seem to have been raised before. In~\cite{BanWan2}, a version of Davis's inequality is proved for orthogonal martingales. With the notation of~\eqref{eq:ortho}, this result leads to the estimate
\begin{equation}\label{eq2:ortho}
\pr\{|\MA \trans M_{\zeta-}|>\lambda\} \le  \frac{D}{\lambda} \, \|\MA\| \, \ex|M_{\zeta-}|.
\end{equation}
If $1<p<\infty$, the conditional expectation is a contraction on $L^p$ and~\eqref{eq:ortho} gives the estimate~\eqref{ortho-A} for the operator ${\CT}_{\!\MA}$. Unfortunately, this reasoning fails for the weak-type inequality \eqref{eq2:ortho} and we cannot conclude the same for ${\CT}_{\!\MA}$.  It is interesting to note here that this is exactly the same situation that arises in the (still open) problem of obtaining the best constant in the weak-type $(1, 1)$ inequality for the Riesz transforms in $\R^n$, $n \ge 2$. Because of the inequalities in~\cite{BanWan2}, we know that an anologue of~\eqref{eq2:ortho} gives the optimal constant for the martingales associated to the Riesz transforms. However, unlike the case of $n=1$, these martingales are not just functions of the exit position of the Brownian motion from the upper half space $\R^{n+1}_{+}$ and the conditional expectation is non-trivial. For more on this, we refer the reader to~\cite{bib:b10}, and especially the discussion preceding \emph{Problem~7}.

In the present case of the discrete Hilbert transform $\dht$ we conjecture that the best constant $C$ in the weak-type $(1, 1)$ inequality
\formula[eq:weak]{
 \#\{n \in \Z : |\dht a_n| > \lambda\} & \le \frac{C}{\lambda} \sum_{n \in \Z} |a_n|
}
is the same as Davis's constant $D$ and that this is the case for all the versions discussed in this paper. This is equivalent to the estimate $C \le D$, because, following the approach of E.~C.~Titchmarsh in~\cite{bib:t28} or E.~Laeng in~\cite{bib:l07} (Theorems 4.2 and 4.3), we easily find that $C \ge D$. For the convenience of the reader, we sketch the proof of this inequality.

It is not difficult to show that if $f$ is a smooth and compactly supported function on $\R$, then
\formula{
 \lim_{\eps \to 0^+} \frac{1}{\pi}\sum_{m \in \Z \setminus \{0\}} \frac{f(y - \eps m)}{m} & = H f(y)
}
for all $x \in \R$. Therefore, by Fatou's lemma,
\formula{
 m \{ x \in \R : |H f(x)| > \lambda\} & \le \liminf_{\eps \to 0^+} m \biggl\{x \in \R: \biggl| \frac{1}{\pi} \sum_{m \in \Z \setminus \{0\}} \frac{f(x - \eps m)}{m} \biggr| > \lambda\biggr\} .
}
It is therefore sufficient to estimate the expression in the right-hand side.

We write $x = \eps y = \eps (z + n) \eps$ for $z \in [0, 1)$ and $n \in \Z$. By Fubini,
\formula{
 & m \biggl\{x \in \R: \biggl| \frac{1}{\pi} \sum_{m \in \Z \setminus \{0\}} \frac{f(x - \eps m)}{m} \biggr| > \lambda\biggr\} \\
 & \hspace*{5em} = \eps \, m \biggl\{y \in \R: \biggl| \frac{1}{\pi} \sum_{m \in \Z \setminus \{0\}} \frac{f(\eps(y - m))}{m} \biggr| > \lambda\biggr\} \\
 & \hspace*{5em} = \eps \int_0^1 \# \biggl\{n \in \Z: \biggl| \frac{1}{\pi} \sum_{m \in \Z \setminus \{0\}} \frac{f(\eps(z + n - m))}{m} \biggr| > \lambda\biggr\} dx .
}
We now apply~\eqref{eq:weak} to the sequence $a_n = f(\eps(z + n))$ with $z$ fixed. It follows that
\formula{
 & m \biggl\{x \in \R: \biggl| \frac{1}{\pi} \sum_{m \in \Z \setminus \{0\}} \frac{f(x - \eps m)}{m} \biggr| > \lambda\biggr\} \\
 & \hspace*{2em} \le \frac{C \eps}{\lambda} \int_0^1 \sum_{n \in \Z} |f(\eps(z + n))| dz = \frac{C \eps}{\lambda} \int_{-\infty}^\infty |f(\eps y)| dy =  \frac{C}{\lambda} \int_{-\infty}^\infty |f(x)| dx .
}
This proves that the estimate~\eqref{davis} holds with constant $C$ for smooth, compactly supported $f$. Extension to general $f \in L^1(\R)$ is standard, and therefore $D \le C$, as desired.

%
%

\section{Technical results}
\label{sec:tech}

In this final section we prove a series of auxiliary lemmas. Some of arguments are variants of those commonly used in the computations of occupation measures for Doob $h$-processes and some are more elementary inequalities and identities. Nevertheless, for the reader's convenience, we present full details. Throughout this section we use the notation introduced in Section~\ref{sec:proof}:
\formula{
 p_n(x, y) & = \frac{1}{\pi} \, \frac{y}{(x - 2 \pi n)^2 + y^2} \, , \\
 h(x, y) & = \frac{1}{2 \pi} \, \frac{\sinh y}{\cosh y - \cos x} \, , \\
 G(x, y) & = \frac{1}{2 \pi} \log \frac{(x - x_0)^2 + (y + y_0)^2}{(x - x_0)^2 + (y - y_0)^2} \, .
}
We begin with justification of the equality in the expression~\eqref{eq:h} for $h(x, y)$.

\begin{lemma}
\label{lem:csc}
For $x \in \R$ and $y > 0$ we have
\formula{
 \sum_{n \in \Z} p_n(x, y) & = \frac{\sinh y}{\cosh y - \cos x} \, .
}
The sum in the left-hand side can be differentiated term by term.
\end{lemma}

\begin{proof}
By formula~1.421.3 in~\cite{bib:gr07}, we have
\formula{
 \lim_{N \to \infty} \sum_{n = -N}^N \frac{1}{z - 2 \pi n} & = \frac{\cot(z/2)}{2} \, .
}
It is easy to see that the convergence is locally uniform in $z \in \C \setminus 2 \pi \Z$, and therefore the limit of partial sums in the left-hand side can be differentiated term by term. If $z = x + i y$, then the imaginary part of the left-hand side is the series $\sum_{n \in \Z} p_n(x, y)$, while the imaginary part of the right-hand side simplifies to $\sinh y / (\cosh y - \cos x)$.
\end{proof}

Next three lemmas evaluate the limit of the expression~\eqref{eq:mdht:nm} for $\mdht_{nm}$ as the starting point $(x_0, y_0)$ tends to infinity.

\begin{lemma}
\label{lem:hest}
For $x \in \R$ and $y > 0$ we have
\formula{
 \frac{1}{2 \pi} \, \frac{y}{y + 1} \le h(x, y) \le \frac{1}{2 \pi} \, \frac{y + 2}{y} \, .
}
\end{lemma}

\begin{proof}
Clearly,
\formula{
 \frac{1}{2 \pi} \, \frac{\sinh y}{\cosh y} \le h(x, y) & \le \frac{1}{2 \pi} \, \frac{\sinh y}{\cosh y - 1} = \frac{1}{2 \pi} \, \frac{\cosh y + 1}{\sinh y} \, .
}
The lower bound follows from an elementary inequality $y \cosh y \le (y + 1) \sinh y$ (which simplifies to $e^{2 y} \ge 1 + 2 y$). For the upper bound, we use the same inequality combined with $1 \le (\sinh y) / y$ to estimate the numerator.
\end{proof}

The following two lemmas are variation of Lemma~IV.3.2 and Lemma~IV.3.2 and~IV.3.3 in~\cite{bib:b95}, respectively.

\begin{lemma}
\label{lem:gest}
Suppose that $x_0 = 0$. For $x \in \R$ and $y > 0$ we have
\formula{
 \lim_{y_0 \to \infty} \frac{G(x, y)}{p_n(0, y_0)} & = 2 y
}
and if $y_0 \ge 2 \pi n$, then
\formula{
 \frac{G(x, y)}{p_n(0, y_0)} & \le y g(y / y_0) ,
}
where $g(t) = t^{-1} \log(1 + 4 t (t - 1)^{-2})$.
\end{lemma}

\begin{proof}
Observe that
\formula{
 \frac{G(x, y)}{p_n(0, y_0)} & = \frac{(2 \pi n)^2 + y_0^2}{2 y_0} \, \log \expr{1 + \frac{4 y_0 y}{(x - x_0)^2 + (y - y_0)^2}} \, .
}
As $y_0 \to \infty$, the logarithm is asymptotically equal to $4 y y_0 / ((x - x_0)^2 + (y - y_0)^2)$, and the first statement follows. Furthermore, if $y_0 \ge 2 \pi n$, we have
\formula{
 \frac{G(x, y)}{p_n(0, y_0)} & \le y_0 \, \log \expr{1 + \frac{4 y_0 y}{(y - y_0)^2}} = y_0 \, \log \expr{1 + \frac{4 (y / y_0)}{(y / y_0 - 1)^2}} \, ,
}
as desired.
\end{proof}

\begin{lemma}
\label{lem:uni}
Suppose that $x_0 = 0$. For $x \in \R$ and $y > 0$ one can take the limit as $y_0 \to \infty$ under the integral sign in the expression
\formula[eq:uni]{
\begin{aligned}
 \int_{-\infty}^\infty \int_0^\infty \frac{G(x, y)}{p_n(0, y_0)} \bigl( & h^{-1}(x, y) \MH\nabla p_m(x, y) \cdot \nabla p_n(x, y) \\
 \phantom{\int_{-\infty}^\infty} {} + {} & p_n(x, y) \MH\nabla p_m(x, y) \cdot \nabla h^{-1}(x, y) \\
 \phantom{\int_{-\infty}^\infty} {} - {} & p_m(x, y) \MH\nabla p_n(x, y)\cdot \nabla h^{-1}(x, y) \bigr) dy dx .
\end{aligned}
}
\end{lemma}

\begin{proof}
If $n = m$, then the integrand is zero (because $\MH\nabla p_n(x, y) \cdot \nabla p_n(x, y) = 0$), and there is nothing to be proved. Therefore, we suppose that $n \ne m$. We gather the necessary estimates. For simplicity, we write $c(n, m)$ for a constant that depends on $n$ and $m$ whenever the value of this constant does not play any role. We assume that $y_0 \ge 2 \pi n$, so that, by Lemma~\ref{lem:gest},
\formula[eq:uni:1]{
 \frac{G(x, y)}{p_n(0, y_0)} & \le y g(y / y_0) ,
}
where $g(t) = t^{-1} \log(1 + 4 t (t - 1)^{-2})$. Note that $g$ is increasing on $(0, 1)$ and decreasing on $(1, \infty)$.

Observe that
\formula{
 |\nabla p_n(x, y)| & = \frac{1}{\pi} \, \frac{1}{(x - 2 \pi n)^2 + y^2} = \frac{p_n(x, y)}{y} \, ,
}
and so, by Lemma~\ref{lem:csc},
\formula{
 |\nabla h(x, y)| & \le \frac{h(x, y)}{y} \, .
}
It follows that
\formula[eq:uni:2]{
\begin{aligned}
 |h^{-1}(x, y) \MH\nabla p_m(x, y) \cdot \nabla p_n(x, y)| & \le \frac{p_n(x, y) p_m(x, y)}{y^2 h(x, y)} , \\
 |p_n(x, y) \MH\nabla p_m(x, y) \cdot \nabla h^{-1}(x, y)| & \le \frac{p_n(x, y) p_m(x, y)}{y^2 h(x, y)} , \\
 |p_m(x, y) \MH\nabla p_n(x, y) \cdot \nabla h^{-1}(x, y)| & \le \frac{p_n(x, y) p_m(x, y)}{y^2 h(x, y)} .
\end{aligned}
}
Since $h(x, y) \ge p_n(x, y) + p_m(x, y)$, we have
\formula[eq:uni:3]{
\begin{aligned}
 \frac{p_n(x, y) p_m(x, y)}{h(x, y)} & \le \frac{p_n(x, y) p_m(x, y)}{p_n(x, y) + p_m(x, y)} \\
 & = \frac{1}{\pi} \, \frac{y}{(x - 2 \pi n)^2 + (x - 2 \pi m)^2 + 2 y^2} \\
 & = \frac{1}{2 \pi} \, \frac{y}{(x - \pi n - \pi m)^2 + y^2 + \pi^2 (n - m)^2} \\
 & \le c(n, m) \, \frac{y}{x^2 + 1} \, .
\end{aligned}
}
Finally, by Lemma~\ref{lem:hest},
\formula[eq:uni:4]{
 \frac{p_n(x, y) p_m(x, y)}{h(x, y)} & \le \frac{2 \pi (y + 1)}{y} \, p_n(x, y) p_m(x, y) \le c(n, m) \, \frac{y (y + 1)}{(x^2 + y^2)^2} \, .
}
Suppose that $y_0 \ge 2$, so that $g(y / y_0) \le g(1/2)$ when $y \in (0, 1)$. Thus, estimates~\eqref{eq:uni:1}, \eqref{eq:uni:2} and~\eqref{eq:uni:3} imply that for $x \in \R$ and $y \in (0, 1)$ the absolute value of the integrand $j(x, y)$ in~\eqref{eq:uni} is bounded above by an integrable function, namely,
\formula{
 |j(x, y)| & \le c(n, m) \, \frac{1}{x^2 + 1} \, .
}
Similarly, for $x \in \R$ and $y \in (1, y_0/2)$ we use $g(y / y_0) \le g(1/2)$ together with estimates~\eqref{eq:uni:1}, \eqref{eq:uni:2} and~\eqref{eq:uni:4} to find that
\formula{
 |j(x, y)| & \le c(n, m) \, \frac{y + 1}{(x^2 + y^2)^2} \le c(n, m) \, \frac{y}{(x^2 + y^2)^2} \, .
}
Observe that the right-hand side is an integrable function of $x \in \R$, $y > 1$. By dominated convergence, the above two estimates allow us to take the limit as $y_0 \to \infty$ under the integral sign in the integral of $j(x, y) \ind_{(0, y_0/2)}(y)$. 

Finally, using again estimates~\eqref{eq:uni:1}, \eqref{eq:uni:2} and~\eqref{eq:uni:4}, we obtain
\formula{
 \int_{-\infty}^\infty \int_{y_0/2}^\infty |j(x, y)| dy dx & \le c(n, m) \int_{-\infty}^\infty \int_{y_0/2}^\infty \frac{y g(y / y_0)}{(x^2 + y^2)^2} \, dy dx \\
 & = c(n, m) \int_{y_0/2}^\infty \frac{g(y / y_0)}{y^2} \, dy = \frac{c(n, m)}{y_0} \int_{1/2}^\infty \frac{g(t)}{t^2} \, dt ,
}
and the right-hand side converges to zero as $y_0 \to \infty$. Therefore, the integral of $j(x, y) \ind_{(y_0/2, \infty)}(y)$ converges to zero as $y_0 \to \infty$. This proves the desired result.
\end{proof}

In the following eight lemmas, we evaluate a number of integrals that appear in the expression~\eqref{eq:mdht:n} for $\mdht_n = \mdht_{n0}$. We use the symbol $:=$ to say that the left-hand side is defined to be equal to the right-hand side.

\begin{lemma}
\label{lem:int:1}
For $y > 0$ and $n \ne 0$ we have
\formula{
 I_1 := 2 \pi \int_{-\infty}^\infty \MH\nabla p_0(x, y) \cdot \nabla p_n(x, y) dx & = \pi n \, \frac{3 y^2 - \pi^2 n^2}{(y^2 + \pi^2 n^2)^3} \, .
}
\end{lemma}

\begin{proof}
Since $p_n(x, y) = p_0(x - 2 \pi n, y)$ and
\formula{
 \nabla p_0(x, y) & = \expr{-\frac{1}{\pi} \, \frac{2 x y}{(x^2 + y^2)^2}, \frac{1}{\pi} \, \frac{x^2 - y^2}{(x^2 + y^2)^2}} ,
}
we have
\formula{
 f(x) & := \MH\nabla p_0(x, y) \cdot \nabla p_n(x, y) \\
 & = \frac{1}{\pi^2} \, \frac{2 (x - 2 \pi n) y (x^2 - y^2) - 2 x y ((x - 2 \pi n)^2 + y^2)}{(x^2 + y^2)^2 ((x - 2 \pi n)^2 + y^2)^2} \\
 & = \frac{1}{\pi^2} \, \frac{4 \pi n y (x^2 + y^2 - 2 \pi n x)}{(x^2 + y^2)^2 ((x - 2 \pi n)^2 + y^2)^2} \, .
}
The above expression defines a meromorphic function $f(x)$ in the upper complex half-plane, which decays as $|x|^{-6}$ when $|x| \to \infty$, and has (double) poles at $i y$ and $2 \pi n + i y$. By an elementary calculation, we find that
\formula{
 f(x) & = \frac{i}{8 \pi^2} \biggl( \frac{1}{(\pi n - i y)^2 (x - i y)^2} + \frac{1}{(\pi n - i y)^3 (x - i y)} \\
 & \hspace*{2.5em} - \frac{1}{(\pi n + i y) (x + i y)^2} - \frac{1}{(\pi n + i y)^3 (x + i y)} \\
 & \hspace*{2.5em} - \frac{1}{(\pi n + i y)^2 (x - 2 \pi n - i y)^2} + \frac{1}{(\pi n + i y)^3 (x - 2 \pi n - i y)} \\
 & \hspace*{2.5em} + \frac{1}{(\pi n - i y)^2 (x - 2 \pi n + i y)^2} - \frac{1}{(\pi n - i y)^3 (x - 2 \pi n + i y)} \biggr) \, .
}
Therefore, by the residue theorem,
\formula{
 I_1 & = 4 \pi^2 i \res(f, i y) + 4 \pi^2 i \res(f, 2 \pi n + i y) \\
 & = \frac{-1}{2 (\pi n - i y)^3} + \frac{-1}{2 (\pi n + i y)^3} = \frac{3 \pi n y^2 - \pi^3 n^3}{(y^2 + \pi^2 n^2)^3} \, ,
}
as desired.
\end{proof}

\begin{lemma}
\label{lem:int:2}
For $y > 0$ and $n \ne 0$ we have
\formula{
 I_2 := 2 \pi \int_{-\infty}^\infty \MH\nabla p_0(x, y) \cdot \nabla p_n(x, y) \cos x \, dx & = \pi n \, \frac{\pi^2 n^2 (2 y - 1) + y^2 (2 y + 3)}{(y^2 + \pi^2 n^2)^3} \, e^{-y} .
}
\end{lemma}

\begin{proof}
Let $f$ be the function defined in the proof of Lemma~\ref{lem:int:1}, and let $g(x) = f(x) e^{i x}$. By the residue theorem,
\formula{
 I_2 & = \re\expr{2 \pi \int_{-\infty}^\infty g(x) dx} = \re \bigl(4 \pi^2 i \res(g, i y) + 4 \pi^2 i \res(g, 2 \pi n + i y)\bigr) \\
 & = \re\expr{\frac{-e^{-y}}{2 (\pi n - i y)^3} + \frac{-i e^{-y}}{2 (\pi n - i y)^2} + \frac{-e^{2 \pi i n - y}}{2 (\pi n + i y)^3} - \frac{-i e^{2 \pi i n - y}}{2 (\pi n + i y)^2}} \\
 & = \re\expr{\frac{3 \pi n y^2 - \pi^3 n^3}{(y^2 + \pi^2 n^2)^3} + \frac{2 \pi n y}{(y^2 + \pi^2 n^2)^2}} e^{-y} \\
 & = \frac{3 \pi n y^2 - \pi^3 n^3 + 2 \pi n y^3 + 2 \pi^3 n^3 y}{(y^2 + \pi^2 n^2)^3} \, e^{-y} ,
}
as desired.
\end{proof}

\begin{lemma}
\label{lem:int:3}
For $y > 0$ and $n \ne 0$ we have
\formula{
 I_3 := 2 \pi \int_{-\infty}^\infty p_n(x, y) (-\partial_y p_0(x, y)) \sin x \, dx & = \frac{1}{2 \pi n} \, \frac{y^2}{y^2 + \pi^2 n^2} \, e^{-y} .
}
\end{lemma}

\begin{proof}
We have
\formula{
 f(x) := p_n(x, y) (-\partial_y p_0(x, y)) e^{i x} = -\frac{1}{\pi^2} \, \frac{y (x^2 - y^2)}{(x^2 + y^2)^2 ((x - 2 \pi n)^2 + y^2)} \, e^{i x} \, ;
}
this formula defines a meromorphic function of $x$ in the upper half-plane which decays as $|x|^{-4}$ when $|x| \to \infty$, and which has a (double) pole at $x = i y$ and a (simple) pole at $x = 2 \pi n + i y$. After elementary calculations, we get
\formula{
 I_3 & = \im\expr{2 \pi \int_{-\infty}^\infty f(x) dx} = \im \bigl(4 \pi^2 i \res(f, i y) + 4 \pi^2 i \res(f, 2 \pi n + i y)\bigr) \\
 & = \im\expr{\frac{2 \pi^2 n^2 y - y^2 - 2 \pi i n (y + 1)}{4 \pi^2 n^2 (\pi n - i y)^2} \, e^{-y} + \frac{y^2 - 2 \pi^2 n^2 - 2 \pi i n y}{4 \pi^2 n^2 (\pi n + i y)^2} \, e^{-y}} \\
 & = \im\expr{\frac{\pi^3 n^3 (y - 1) + \pi n y^2 (y + 1) + i y^2 (y^2 + \pi^2 n^2)}{2 \pi n (y^2 + \pi^2 n^2)^2}} e^{-y} \\
 & = \frac{y^2}{2 \pi n (y^2 + \pi^2 n^2)} \, e^{-y} ,
}
as desired.
\end{proof}

\begin{lemma}
\label{lem:int:4}
For $y > 0$ and $n \ne 0$ we have
\formula{
 I_4 := 2 \pi \int_{-\infty}^\infty p_n(x, y) \partial_x p_0(x, y) dx & = -\frac{\pi n y}{(y^2 + \pi^2 n^2)^2} \, .
}
\end{lemma}

\begin{proof}
Observe that
\formula{
 f(x) := p_n(x, y) \partial_x p_0(x, y) = -\frac{1}{\pi^2} \, \frac{2 x y^2}{(x^2 + y^2)^2 ((x - 2 \pi n)^2 + y^2)} \, ;
}
again this formula defines a meromorphic function of $x$ in the upper half-plane which decays as $|x|^{-5}$ when $|x| \to \infty$, and which has a (double) pole at $x = i y$ and a (simple) pole at $x = 2 \pi n + i y$. After elementary calculations, we have
\formula{
 I_4 & = 2 \pi \int_{-\infty}^\infty f(x) dx \\
 & = 4 \pi^2 i \res(f, i y) + 4 \pi^2 i \res(f, 2 \pi n + i y) \\
 & = \frac{-2 \pi n y + i y^2}{4 \pi^2 n^2 (\pi n - i y)^2} + \frac{-2 \pi n y - i y^2}{4 \pi^2 n^2 (\pi n + i y)^2} \\
 & = -\frac{\pi n y}{(y^2 + \pi^2 n^2)^2} \, ,
}
as desired.
\end{proof}

\begin{lemma}
\label{lem:int:5}
For $y > 0$ and $n \ne 0$ we have
\formula{
 I_5 := 2 \pi \int_{-\infty}^\infty p_n(x, y) (-\partial_x p_0(x, y)) \cos x \, dx & = \frac{1}{2 \pi n} \, \frac{y^4 + \pi^2 n^2 y (y - 2)}{(y^2 + \pi^2 n^2)^2} \, e^{-y} .
}
\end{lemma}

\begin{proof}
Let $f$ be the function defined in the proof of Lemma~\ref{lem:int:4}, and let $g(x) = f(x) e^{i x}$. After elementary calculations, we obtain
\formula{
 I_5 & = \re\expr{2 \pi \int_{-\infty}^\infty g(x) dx} = \re \bigl(4 \pi^2 i \res(g, i y) + 4 \pi^2 i \res(g, 2 \pi n + i y)\bigr) \\
 & = \re\expr{\frac{-2 \pi n y (y + 1) + i y (y - 2 \pi^2 n^2)}{4 \pi^2 n^2 (\pi n - i y)^2} \, e^{-y} + \frac{-2 \pi n y - i y^2}{4 \pi^2 n^2 (\pi n + i y)^2} \, e^{-y}} \\
 & = \re\expr{\frac{y^4 + \pi^2 n^2 y (y - 2) - i \pi n y (y^2 + \pi^2 n^2)}{2 \pi n (y^2 + \pi^2 n^2)^2}} e^{-y} \\
 & = \frac{y^4 + \pi^2 n^2 y (y - 2)}{2 \pi n (y^2 + \pi^2 n^2)^2} \, e^{-y} ,
}
as desired.

\end{proof}

\begin{lemma}
\label{lem:int:6}
For $y > 0$ and $n \ne 0$ we have
\formula{
 & \int_{-\infty}^\infty h^{-1}(x, y) \MH\nabla p_0(x, y) \cdot \nabla p_n(x, y) dx \\
 & \hspace*{6em} + 2 \int_{-\infty}^\infty p_n(x, y) \MH\nabla p_0(x, y) \cdot \nabla h^{-1}(x, y) dx \\
 & \hspace*{12em} = \frac{\pi n (3 y^2 - \pi^2 n^2)}{(y^2 + \pi^2 n^2)^3} + \frac{y^2}{\pi n (y^2 + \pi^2 n^2)^3 \sinh^2 y} \, .
}
\end{lemma}

\begin{proof}
Recall that $h^{-1}(x, y) = 2 \pi (\cosh y - \cos x) / \sinh y$, and observe that
\formula{
 \nabla h^{-1}(x, y) & = 2 \pi \expr{\frac{\sin x}{\sinh y} \, , \, \frac{\cos x \cosh y - 1}{\sinh^2 y}} .
}
Therefore, with the notation introduced in the previous lemmas,
\formula{
 I_6 & := \int_{-\infty}^\infty h^{-1}(x, y) \MH\nabla p_0(x, y) \cdot \nabla p_n(x, y) dx \\
 & = \frac{\cosh y}{\sinh y} \, I_1 - \frac{1}{\sinh y} \, I_2 ,
}
and
\formula{
 I_7 & := \int_{-\infty}^\infty p_n(x, y) \MH\nabla p_0(x, y) \cdot \nabla h^{-1}(x, y) dx \\
 & = \frac{1}{\sinh y} \, I_3 + \frac{\cosh y}{\sinh^2 y} \, I_5 - \frac{1}{\sinh^2 y} \, I_4 .
}
By Lemmas~\ref{lem:int:1} and~\ref{lem:int:2},
\formula{
 I_6 & = \frac{\cosh y}{\sinh y} \, \pi n \, \frac{3 y^2 - \pi^2 n^2}{(y^2 + \pi^2 n^2)^3} - \frac{1}{\sinh y} \, \pi n \, \frac{\pi^2 n^2 (2 y - 1) + y^2 (2 y + 3)}{(y^2 + \pi^2 n^2)^3} \, e^{-y} \\
 & = \frac{\pi n (3 y^2 - \pi^2 n^2) \cosh y - \pi n (\pi^2 n^2 (2 y - 1) + y^2 (2 y + 3)) e^{-y}}{(y^2 + \pi^2 n^2)^3 \sinh y} \\
 & = \frac{\pi n (\pi^2 n^2 (2 y - 1) + y^2 (2 y + 3))}{(y^2 + \pi^2 n^2)^3} \\
 & \hspace*{5em} + \frac{\pi n ((3 y^2 - \pi^2 n^2) - \pi^2 n^2 (2 y - 1) - y^2 (2 y + 3)) \cosh y}{(y^2 + \pi^2 n^2)^3 \sinh y} \\
 & = \frac{\pi n (2 \pi^2 n^2 y - \pi^2 n^2 + 2 y^3 + 3 y^2)}{(y^2 + \pi^2 n^2)^3} - \frac{2 \pi n y \cosh y}{(y^2 + \pi^2 n^2)^2 \sinh y}
}
(we used the identity $e^{-y} = \cosh y - \sinh y$), while by Lemmas~\ref{lem:int:3}, \ref{lem:int:4} and~\ref{lem:int:5},
\formula{
 I_7 & = \frac{1}{\sinh y} \, \frac{1}{2 \pi n} \, \frac{y^2}{y^2 + \pi^2 n^2} \, e^{-y} \\
 & \hspace*{5em} + \frac{\cosh y}{\sinh^2 y} \, \frac{1}{2 \pi n} \, \frac{y^4 + \pi^2 n^2 (y^2 - 2 y)}{(y^2 + \pi^2 n^2)^2} \, e^{-y} + \frac{1}{\sinh^2 y} \, \frac{\pi n y}{(y^2 + \pi^2 n^2)^2} \\
 & = \frac{y^2 (y^2 + \pi^2 n^2) e^{-y} \sinh y + (y^4 + \pi^2 n^2 (y^2 - 2 y)) e^{-y} \cosh y + 2 \pi^2 n^2 y}{2 \pi n (y^2 + \pi^2 n^2)^2 \sinh^2 y} \\
 & = \frac{-y^2 (y^2 + \pi^2 n^2) + (y^4 + \pi^2 n^2 (y^2 - 2 y))}{2 \pi n (y^2 + \pi^2 n^2)^2} \\
 & \hspace*{5em} + \frac{(y^2 (y^2 + \pi^2 n^2) - (y^4 + \pi^2 n^2 (y^2 - 2 y))) \cosh y}{2 \pi n (y^2 + \pi^2 n^2)^2 \sinh y} \\
 & \hspace*{5em} + \frac{(y^4 + \pi^2 n^2 (y^2 - 2 y)) + 2 \pi^2 n^2 y}{2 \pi n (y^2 + \pi^2 n^2)^2 \sinh^2 y} \\
 & = -\frac{\pi n y}{(y^2 + \pi^2 n^2)^2} + \frac{\pi n y \cosh y}{(y^2 + \pi^2 n^2)^2 \sinh y} + \frac{y^2}{2 \pi n (y^2 + \pi^2 n^2) \sinh^2 y}
}
(we used the identities $e^{-y} \sinh y = \cosh y \sinh y - \sinh^2 y$ and $e^{-y} \cosh y = \cosh^2 y - \cosh y \sinh y = 1 + \sinh^2 y - \cosh y \sinh y$). It follows that
\formula{
 I_6 + 2 I_7 & = \frac{\pi n (2 \pi^2 n^2 y - \pi^2 n^2 + 2 y^3 + 3 y^2)}{(y^2 + \pi^2 n^2)^3} \\
 & \hspace*{5em} - \frac{2 \pi n y}{(y^2 + \pi^2 n^2)^2} + \frac{y^2}{\pi n (y^2 + \pi^2 n^2) \sinh^2 y} \\
 & = \frac{\pi n (-\pi^2 n^2 + 3 y^2)}{(y^2 + \pi^2 n^2)^3} + \frac{y^2}{\pi n (y^2 + \pi^2 n^2) \sinh^2 y} \, ,
}
as desired.
\end{proof}

\begin{lemma}
\label{lem:int:7}
For $n \ne 0$ we have
\formula{
 \int_0^\infty 2 y \, \frac{\pi n (3 y^2 - \pi^2 n^2)}{(y^2 + \pi^2 n^2)^3} \, dy & = \frac{1}{\pi n} \, .
}
\end{lemma}

\begin{proof}
Substituting $y^2 = t$, we obtain
\formula{
 \int_0^\infty 2 y \, \frac{\pi n (3 y^2 - \pi^2 n^2)}{(y^2 + \pi^2 n^2)^3} \, dy & = \int_0^\infty \frac{\pi n (3 t - \pi^2 n^2)}{(\pi^2 n^2 + t)^3} \, dt \\
 & = \int_0^\infty \frac{3 \pi n}{(\pi^2 n^2 + t)^2} \, dt - \int_0^\infty \frac{4 \pi^3 n^3}{(\pi^2 n^2 + t)^3} \, dt \\
 & = \frac{3 \pi n}{\pi^2 n^2} - \frac{2 \pi^3 n^3}{\pi^4 n^4} = \frac{1}{\pi n} \, ,
}
as desired.
\end{proof}

\begin{lemma}
\label{lem:int}
For $n \ne 0$ we have
\formula{
 & \int_{-\infty}^\infty \int_0^\infty 2 y h^{-1}(x, y) \MH\nabla p_0(x, y) \cdot \nabla p_n(x, y) dy dx \\
 & \hspace*{6em} + \int_{-\infty}^\infty \int_0^\infty 4 y p_n(x, y) \MH\nabla p_0(x, y) \cdot \nabla h^{-1}(x, y) dy dx \\
 & \hspace*{12em} = \frac{1}{\pi n} \expr{1 + \int_0^\infty \frac{2 y^3}{(y^2 + \pi^2 n^2) \sinh^2 y} \, dy} .
}
\end{lemma}

\begin{proof}
By Lemma~\ref{lem:int:6}, the left-hand side of the identity that we are proving is equal to
\formula{
 \int_0^\infty 2 y \, \frac{\pi n (3 y^2 - \pi^2 n^2)}{(y^2 + \pi^2 n^2)^3} \, dy + \int_0^\infty 2 y \, \frac{y^2}{\pi n (y^2 + \pi^2 n^2)^3 \sinh^2 y} \, dy .
}
It remains to apply Lemma~\ref{lem:int:7}.
\end{proof}

The remaining three lemmas evaluate the kernel $(\CE_n)$ in the relation~\eqref{eq:ihj} between the discrete Hilbert transform $\dht$ and the operator $\mdht$ obtained using the martingale transform.

\begin{lemma}
\label{lem:hp}
For $y > 0$, the discrete Hilbert transform of the sequence
\formula{
 \CA_n(y) & = \frac{1}{y^2 + \pi^2 n^2}
}
is given by
\formula{
 \CB_n(y) & = \frac{\pi n \coth y}{y (y^2 + \pi^2 n^2)} - \frac{2 \pi n}{(y^2 + \pi^2 n^2)^2} \, .
}
\end{lemma}

\begin{proof}
Let $(\CB_n(y))$ be the discrete Hilbert transform of $(\CA_n(y))$. For $m > 0$ we have
\formula{
 & \frac{1}{m} \, \frac{1}{y^2 + \pi^2 (n - m)^2} - \frac{1}{m} \, \frac{1}{y^2 + \pi^2 (n + m)^2} \\
 & \hspace*{5em} = \frac{4 \pi^2 n}{(y^2 + \pi^2 (n - m)^2) (y^2 + \pi^2 (n + m)^2)} \\
 & \hspace*{5em} = \frac{\pi^2}{y^2 + \pi^2 n^2} \expr{\frac{2 n - m}{y^2 + \pi^2 (n - m)^2} + \frac{2 n + m}{y^2 + \pi^2 (n + m)^2}} .
}
Therefore,
\formula{
 \pi \CB_n(y) & = \sum_{m = 1}^\infty \frac{1}{m} \expr{\frac{1}{y^2 + \pi^2 (n - m)^2} - \frac{1}{y^2 + \pi^2 (n + m)^2}} \\
 & = \frac{\pi^2}{y^2 + \pi^2 n^2} \sum_{m = 1}^\infty \expr{\frac{2 n - m}{y^2 + \pi^2 (n - m)^2} + \frac{2 n + m}{y^2 + \pi^2 (n + m)^2}} .
}
The sum in the right-hand side is equal to
\formula{
 & \lim_{M \to \infty} \sum_{m = 1}^M \expr{\frac{2 n - m}{y^2 + \pi^2 (n - m)^2} + \frac{2 n + m}{y^2 + \pi^2 (n + m)^2}} \\
 & \hspace*{5em} = -\frac{2 n}{y^2 + \pi^2 n^2} + \lim_{M \to \infty} \sum_{m = -M}^M \frac{2 n - m}{y^2 + \pi^2 (n - m)^2} \\
 & \hspace*{5em} = -\frac{2 n}{y^2 + \pi^2 n^2} + \lim_{M \to \infty} \sum_{j = n - M}^{n + M} \frac{n - j}{y^2 + \pi^2 j^2} \\
 & \hspace*{5em} = -\frac{2 n}{y^2 + \pi^2 n^2} + \lim_{M \to \infty} \sum_{j = -M}^M \frac{n - j}{y^2 + \pi^2 j^2} \\
 & \hspace*{5em} = -\frac{2 n}{y^2 + \pi^2 n^2} + \lim_{M \to \infty} \sum_{j = -M}^M \frac{n}{y^2 + \pi^2 j^2} \, .
}
Therefore,
\formula{
 \pi \CB_n(y) & = \frac{\pi^2}{y^2 + \pi^2 n^2} \expr{-\frac{2 n}{y^2 + \pi^2 n^2} + \sum_{j = -\infty}^\infty \frac{n}{y^2 + \pi^2 j^2}} \, .
}
This is equivalent to the desired result, because by formula~1.421.4 in~\cite{bib:gr07} we have
\formula{
 \sum_{j = -\infty}^\infty \frac{1}{y^2 + \pi^2 j^2} & = \frac{\coth y}{y} \, . \qedhere
}
\end{proof}

\begin{lemma}
\label{lem:ihq}
For $y > 0$, the discrete Hilbert transform of the sequence
\formula{
 \CC_n(y) & = \int_0^y \frac{t \sinh t}{t^2 + \pi^2 n^2} \, dt
}
is given by
\formula{
 \CD_n(y) & = \frac{1}{\pi n} \expr{\sinh y - \frac{y^2 \sinh y}{y^2 + \pi^2 n^2}}
}
for $n \ne 0$, and $D_0(y) = 0$.
\end{lemma}

\begin{proof}
We use the notation of Lemma~\ref{lem:hp}. By that result, after rearrangement, for $n \ne 0$ we have
\formula{
 \pi n \CB_n(y) & = \frac{\pi^2 n^2 \coth y}{y (y^2 + \pi^2 n^2)} - \frac{2 \pi^2 n^2}{(y^2 + \pi^2 n^2)^2} \\
 & = \frac{\coth y}{y} + \frac{2 y^2}{(y^2 + \pi^2 n^2)^2} - \frac{2 + y \coth y}{y^2 + \pi^2 n^2} \, .
}
It follows that
\formula{
 y \sinh y \, \CB_n(y) & = \frac{1}{\pi n} \expr{\cosh y + \frac{2 y^3 \sinh y}{(y^2 + \pi^2 n^2)^2} - \frac{2 y \sinh y + y^2 \cosh y}{y^2 + \pi^2 n^2)}} \\
 & = \frac{1}{\pi n} \expr{\cosh y - \frac{d}{dy} \, \frac{y^2 \sinh y}{y^2 + \pi^2 n^2}} .
}
We replace $y$ by $t$ and integrate with respect to $t$ over $(0, y)$ to find that
\formula{
 \int_0^y t \sinh t \, \CB_n(t) dt & = \frac{1}{\pi n} \expr{\sinh y - \frac{y^2 \sinh y}{y^2 + \pi^2 n^2}} .
}
Recall that $(\CB_n(t))$ is the discrete Hilbert transform of $(\CA_n(t))$. The proof will be complete if we prove that in the left-hand side we may exchange the integral with the discrete Hilbert transform, that is, if we show that
\formula{
 \int_0^y \pi t \sinh t \expr{\sum_{m \in \Z \setminus \{0\}} \frac{\CA_{n - m}(t)}{\pi m}} dt & = \sum_{m \in \Z \setminus \{0\}} \frac{1}{\pi m} \int_0^y \pi t \sinh t \, \CA_{n - m}(t) dt .
}
However, the above identity is a simple consequence of Fubini's theorem: we have
\formula{
 \sum_{m \in \Z \setminus \{0\}} \abs{\frac{\CA_{n - m}(t)}{\pi m}} & \le \sum_{m \in \Z} |\CA_{n - m}(t)| = \CA_0(t) + \sum_{m \in \Z \setminus \{0\}} |\CA_m(t)| \\
 & \le \frac{1}{t^2} + \sum_{m \in \Z \setminus \{0\}} \frac{1}{\pi^2 m^2} = \frac{1}{t^2} + \frac{1}{3} \, ,
}
and $t (t^{-2} + 1/3) \sinh t$ is integrable over $(0, y)$.
\end{proof}

\begin{lemma}
\label{lem:ihj}
Let
\formula{
 \CE_n & = -\int_0^\infty \frac{2 y}{\sinh^3 y} \expr{\int_0^y \frac{t \sinh t}{t^2 + \pi^2 n^2} \, dt} dy ,
}
for $n \ne 0$, and
\formula{
 \CE_0 & = \int_0^\infty \frac{2 y}{\sinh^3 y} \expr{\sinh y - \int_0^y \frac{\sinh t}{t} \, dt} dy .
}
Then the discrete Hilbert transform of the sequence $(\CE_n)$ is $(\CF_n)$, where
\formula{
 \CF_n & = \frac{1}{\pi n} \int_0^\infty \frac{2 y^3}{(y^2 + \pi^2 n^2) \sinh^2 y} \, dy
}
for $n \ne 0$, and $\CF_0 = 0$. Furthermore, $\CE_0 > 0$, $\CE_n < 0$ for $n \ne 0$, and the sequence $\CE_n$ is absolutely summable, with sum equal to zero.
\end{lemma}

\begin{proof}
Clearly, $\CE_n < 0$ for $n \ne 0$. Furthermore, since
\formula{
 \sum_{n \in \Z \setminus \{0\}} \frac{1}{t^2 + \pi^2 n^2} & = \frac{\coth t}{t} - \frac{1}{t^2} \, ,
}
for $t > 0$ (by formula~1.421.4 in~\cite{bib:gr07}), we have, by Fubini's theorem,
\formula{
 -\sum_{n \in \Z \setminus \{0\}} \CE_n & = \int_0^\infty \frac{2 y}{\sinh^3 y} \expr{\int_0^y t \sinh t \expr{\frac{\coth t}{t} - \frac{1}{t^2}} dt} dy \\
 & = \int_0^\infty \frac{2 y}{\sinh^3 y} \expr{\int_0^y \expr{\cosh t - \frac{\sinh t}{t}} dt} dy = \CE_0 .
}
This proves the final part of the lemma. We proceed to the proof that the discrete Hilbert transform of $\CE_n$ is $\CF_n$. Since $\CE_{-m} = \CE_m$, we have $\sum_{m \in \Z \setminus \{0\}} (\pi m)^{-1} \CE_{-m} = 0 = \CF_0$ (summability follows from the last part of the proof). Therefore, it suffices to prove that $\CF_n$ is the $n$-th element of the discrete Hilbert transform of $(\CE_n)$ for $n \ne 0$.

Let $(\CC_n(y))$ and $(\CD_n(y))$ be the sequences introduced in Lemma~\ref{lem:ihq}. By that result, $(\CD_n(y))$ is the discrete Hilbert transform of $(\CC_n(y))$. We will now multiply this identity by $2 y \sinh^{-3} y$ and integrate over $y > 0$.

Observe that for $n \ne 0$ we have
\formula{
 \int_0^\infty \frac{2 y}{\sinh^3 y} \expr{\frac{\sinh y}{\pi n} - \CD_n(y)} dy & = \frac{1}{\pi n} \int_0^\infty \frac{y^3}{(y^2 + \pi^2 n^2) \sinh^2 y} \, dy = \CF_n .
}
On the other hand, for $n \ne 0$,
\formula{
 \int_0^\infty \frac{2 y}{\sinh^3 y} \, (-\CC_n(y)) dy & = -\int_0^\infty \frac{2 y}{\sinh^3 y} \expr{\int_0^y \frac{t \sinh t}{t^2 + \pi^2 n^2} \, dt} dy = \CE_n ,
}
while for $n = 0$,
\formula{
 \int_0^\infty \frac{2 y}{\sinh^3 y} \, (\sinh y - \CC_0(y)) dy & = \int_0^\infty \frac{2 y}{\sinh^3 y} \expr{\sinh y - \int_0^y \frac{\sinh t}{t} \, dt} dy = \CE_0 .
}
It follows that
\formula{
 \sum_{m \in \Z \setminus \{0\}} \frac{\CE_{n - m}}{\pi m} & = \sum_{m \in \Z \setminus \{0\}} \frac{1}{\pi m} \int_0^\infty \frac{2 y}{\sinh^3 y} \, (\delta_{nm} \sinh y - \CC_{n - m}(y)) dy
}
where $\delta_{nm} = 1$ if $n = m$ and $\delta_{nm} = 0$ otherwise. Suppose now that we can apply Fubini's theorem and change the order of the integral and the discrete Fourier transform in the right-hand side. Then it follows that for $n \ne 0$,
\formula{
 \sum_{m \in \Z \setminus \{0\}} \frac{\CE_{n - m}}{\pi m} & = \int_0^\infty \expr{\sum_{m \in \Z \setminus \{0\}} \frac{2 y}{\sinh^3 y} \expr{\delta_{nm} \frac{\sinh y}{\pi m} - \frac{\CC_{n - m}(y)}{\pi m}}} dy \\
 & = \int_0^\infty \frac{2 y}{\sinh^3 y} \expr{\frac{\sinh y}{\pi n} - \sum_{m \in \Z \setminus \{0\}} \frac{\CC_{n - m}(y)}{\pi m}} dy = \CF_n ,
}
as desired. It remains to justify the use of Fubini's theorem in the above calculation.

When $n - m \ne 0$ we have
\formula{
 0 \le \CC_{n - m}(y) & = \int_0^y \frac{t \sinh t}{t^2 + \pi^2 (n - m)^2} \, dt \le \frac{y \cosh y - \sinh y}{\pi^2 (n - m)^2} \, ,
}
and therefore
\formula{
 \int_0^\infty \frac{2 y}{\sinh^3 y} \CC_{n - m}(y) dy & \le \frac{1}{\pi^2 (n - m)^2} \int_0^\infty \frac{2 y (y \cosh y - \sinh y)}{\sinh^3 y} \, dy .
}
On the other hand,
\formula{
 0 \le \sinh y - \CC_0(y) & = \sinh y - \int_0^y \frac{\sinh t}{t} \, dt \le \sinh y - y ,
}
so that
\formula{
 \int_0^\infty \frac{2 y}{\sinh^3 y} (\sinh y - \CC_0(y)) dy & \le \int_0^\infty \frac{2 y (\sinh y - y)}{\sinh^3 y} \, dy .
}
It follows that
\formula{
 & \sum_{m \in \Z \setminus \{0\}} \int_0^\infty \abs{\frac{1}{\pi m} \frac{2 y}{\sinh^3 y} \, (\delta_{nm} \sinh y - \CC_{n - m}(y))} dy \\
 & \le \int_0^\infty \frac{2 y (\sinh y - y)}{\sinh^3 y} \, dy + \sum_{m \in \Z \setminus \{0, n\}} \frac{1}{\pi^2 (n - m)^2} \int_0^\infty \frac{2 y (y \cosh y - \sinh y)}{\sinh^3 y} \, dy
}
is finite, and the proof is complete.
\end{proof}

\bigskip

{\bf Acknowledgements:} We thank Eero Saksman for an inspiring discussion in B\k{e}dlewo at the \emph{Probability and Analysis 2017} conference. We also thank Jos\'e Luis Torrea for pointing out reference  \cite{CiaGilRonTor} to us.   

%
%

%
%

\end{document}